\documentclass[11pt]{amsart}

\usepackage{amsmath,amssymb,mathtools}
\usepackage{booktabs}
\usepackage{array}
\usepackage{microtype}
\usepackage{enumitem}
\usepackage{aliascnt}
\usepackage{hyperref}
\usepackage[nameinlink,noabbrev]{cleveref}
\usepackage{tikz}
\usetikzlibrary{arrows.meta,calc}

\definecolor{mapblue}{RGB}{42,67,122}
\definecolor{mapfill}{RGB}{235,239,247}
\definecolor{mapaccent}{RGB}{150,48,55}
\definecolor{mapgray}{RGB}{103,110,120}

\hypersetup{colorlinks=true,linkcolor=blue,citecolor=blue,urlcolor=blue}
\newtheorem{theorem}{Theorem}[section]
\newaliascnt{proposition}{theorem}
\newtheorem{proposition}[proposition]{Proposition}
\aliascntresetthe{proposition}
\newaliascnt{corollary}{theorem}
\newtheorem{corollary}[corollary]{Corollary}
\aliascntresetthe{corollary}
\newaliascnt{lemma}{theorem}
\newtheorem{lemma}[lemma]{Lemma}
\aliascntresetthe{lemma}
\theoremstyle{remark}
\newaliascnt{remark}{theorem}
\newtheorem{remark}[remark]{Remark}
\aliascntresetthe{remark}

\newcommand{\R}{\mathbb{R}}
\newcommand{\C}{\mathbb{C}}
\newcommand{\B}{\mathbb{B}}
\newcommand{\PiC}{\Pi^{\mathbb C}}
\newcommand{\dd}{\,\mathrm d}
\newcommand{\norm}[1]{\lVert #1\rVert}
\newcommand{\abs}[1]{\lvert #1\rvert}
\newcommand{\even}{\mathrm{even}}
\newcommand{\disk}{\mathrm{disk}}
\newcommand{\cone}{\mathrm{cone}}
\newcommand{\pr}{\mathrm{prod}}
\DeclareMathOperator{\ran}{ran}

\title[Positive cubature compression on a conic surface]{Positive Cubature Compression and Optimal Hyperinterpolation Stability on a Conic Surface}

\author{Congpei An}
\email{andbachcp@gmail.com}
\author{Yan Ge}
\email{yge@ncut.edu.cn}


\begin{document}

\begin{abstract}
We study the computational cost and stability of positive cubature and hyperinterpolation on a truncated conic surface with a Jacobi-type weight.  A classical quadratic disk-to-cone map is used in a basis-free form as a measure-preserving $\mathbb Z_2$ quotient.  It identifies the full degree-$m$ conic trace space with the even disk polynomials of degree at most $2m$ and converts positive degree-$2n$ cone cubature into centrally symmetric positive degree-$(4n+1)$ disk cubature, and conversely.  This yields an exact transfer of M\"oller's lower bound and of the node excess above it, so that near-minimal disk formulas produce compressed non-product cone rules.

The same quotient transfers reproducing kernels and hyperinterpolation operators.  Every positive degree-$2n$ cone rule gives an exact $L^2$ sampling isometry on $\Pi_n(V)$, {\color{black}the space of restrictions to $V$ of polynomials of total degree at most $n$}; in particular, the weighted sampling matrix has condition number one, independently of the number and geometry of the nodes.  For  $\gamma=0$, if $\Lambda_n^V$ denotes the $C(V)\to C(V)$ Lebesgue constant of degree-$n$ hyperinterpolation, then every positive degree-$2n$ cone cubature rule satisfies the rule-independent {\color{black}two-sided bound}
$$
       c n \le \Lambda_n^V \le C n.
$$
More generally, {\color{black}every positive degree-$2n$ cone rule satisfies the rule-independent lower bound $\Lambda_{n,\gamma}^V\ge c_\gamma n^{\gamma+1}$ for $\gamma>-1$, while} the upper bound $\Lambda_{n,\gamma}^V \le C_\gamma n^{\gamma+1}$ holds when $\gamma$ is a nonnegative half-integer.  {\color{black}Thus the two-sided order $\Lambda_{n,\gamma}^V\asymp n^{\gamma+1}$ is sharp throughout the nonnegative half-integer range.}  Hence node compression preserves exact $L^2$ conditioning while, {\color{black}in the nonnegative half-integer range,} the associated $L^\infty$ stability has the optimal universal order.  Low-degree near-minimal rules and higher-degree optimized disk schemes illustrate the reduction in sampling cost.
\end{abstract}

\maketitle

\section{Introduction}

Positive cubature on algebraic surfaces plays a dual computational role: it evaluates integrals and it provides discrete sampling for polynomial approximation.  These two tasks are tightly coupled in hyperinterpolation, where the cubature rule determines both the sampling cost and the stability of the resulting polynomial operator.  This paper addresses the following question on a truncated cone:

\begin{quote}
\textit{How far can a positive degree-$2n$ cubature rule be compressed toward the dimension limit while retaining exact $L^2$ sampling and stable degree-$n$ hyperinterpolation?}
\end{quote}

Let
\begin{equation}\label{eq:cone}
V:=\{(x,y,t)\in\R^3:x^2+y^2=t^2,\ 0\le t\le1\},
\end{equation}
parametrized by $(x,y,t)=(t\cos\phi,t\sin\phi,t)$, and equip $V$ with the probability measure
\begin{equation}\label{eq:nu}
\dd\nu_\gamma(t,\phi)=\frac{\gamma+1}{2\pi}(1-t)^\gamma\,\dd t\,\dd\phi,
\qquad \gamma>-1.
\end{equation}
Since $\dd\sigma=\sqrt2\,t\,\dd t\,\dd\phi$, this is the normalized form of
$t^{-1}(1-t)^\gamma\dd\sigma$.

A standard Gauss--Jacobi/trapezoidal rule exact on $\Pi_{2n}(V)$ uses
\begin{equation}\label{eq:product-count-intro}
N_n^{\pr}=(n+1)(2n+1)
\end{equation}
nodes, whereas
\begin{equation}\label{eq:dim-intro}
D_n:=\dim\Pi_n(V)=(n+1)^2,
\end{equation}
and positivity together with degree-$2n$ exactness forces $N\ge D_n$.  Thus the product rule has asymptotic oversampling ratio two.  At the same time, degree-$2n$ exactness implies the exact Marcinkiewicz--Zygmund identity
$$
   \sum_j\lambda_j\abs{p(\xi_j)}^2=\norm{p}_{L^2(\nu_\gamma)}^2,
   \qquad p\in\Pi_n(V),
$$
so the weighted sampling matrix is an isometry.  This already suggests that reducing the node count need not damage $L^2$ conditioning.

The approximation theory of conic and other quadratic surfaces is well developed.  Orthogonal polynomials, Fourier expansions, cubature, and fast approximation were studied by Olver and Xu~\cite{OlverXu2020} and Xu~\cite{XuCone2020,XuFrame2021}, with further approximation results on conic domains in~\cite{GeXu2026}.  Hyperinterpolation was introduced by Sloan~\cite{Sloan1995}; for a recent overview of its quadrature foundations, Marcinkiewicz--Zygmund formulations, variants, and applications, see An, Ran and Wu~\cite{AnRanWu2025}.  On the unit disk, Hansen, Atkinson and Chien~\cite{HansenAtkinsonChien2009} obtained an $O(n\log n)$ bound for a concrete product construction, and Wade~\cite{Wade2013} established linear-order bounds for disk hyperinterpolation with positive cubature weights.  Related spherical and ball results include~\cite{Reimer2000,SloanWomersley2000,WangEtAl2015}; Marcinkiewicz--Zygmund sampling is treated, for example, in~{\cite{AnWu2024,FilbirMhaskar2011}}.  Near-minimal cubature on the disk is governed by M\"oller's lower bound~\cite{Moller1976}; see also~\cite{BenouahmaneCuytYaman2019,Cools1997,CoolsMysovskikhSchmid2001,CoolsVerlinden1992} and, for numerical optimization,~\cite{TakakiForbesRolland2022}.

Positive cubature also has a general compression theory rooted in Tchakaloff's theorem~\cite{Tchakaloff1957}; see Bayer and Teichmann~\cite{BayerTeichmann2006} for a modern proof.  Carath\'eodory--Tchakaloff subsampling gives constructive compression of discrete measures on polynomial spaces, including compact sets and manifolds~\cite{PiazzonSommarivaVianello2017}, while provably positive and exact cubature algorithms have been developed in a broader computational setting~\cite{Glaubitz2023}.  Such general principles control the number of nodes by the dimension of the full exactness space, which here is
$$
   \dim\Pi_{2n}(V)=(2n+1)^2.
$$
The problem studied below is more restrictive and more quantitative: we seek positive degree-$2n$ rules approaching the much smaller square-exactness/M\"oller scale $D_n=\dim\Pi_n(V)=(n+1)^2$, and at the same time ask whether this stronger compression preserves the conditioning and uniform stability of the induced approximation operator.

{\color{black}Let
$$
\mathbb B^2:=\{(u,v)\in\mathbb R^2:u^2+v^2\le1\}.
$$}Our mechanism is the quadratic map
\begin{equation}\label{eq:Q}
Q(u,v):=(u^2-v^2,2uv,u^2+v^2),\qquad (u,v)\in\B^2.
\end{equation}
Writing $z=u+iv$,
$$
Q(z)=(\operatorname{Re}z^2,\operatorname{Im}z^2,\abs z^2),
$$
so $Q(z)=Q(-z)$; in polar coordinates $t=r^2$ and $\phi=2\theta$.The geometry of this antipodal quotient is shown in \Cref{fig:quadratic-quotient}.

\begin{figure}[tbp]
\centering
\begin{tikzpicture}[
  x=1cm,y=1cm,
  line cap=round,line join=round,
  >={Latex[length=2.6mm,width=1.8mm]},
  map line/.style={draw=mapblue,line width=0.75pt},
  guide/.style={draw=mapgray!70,line width=0.45pt},
  point/.style={circle,fill=mapaccent,draw=white,line width=0.35pt,
                inner sep=1.65pt},
  panel title/.style={font=\small}
]
  \coordinate (D) at (-3.25,0);
  \filldraw[map line,fill=mapfill] (D) circle[radius=1.18];
  \draw[guide] ($(D)+(-1.32,0)$)--($(D)+(1.32,0)$)
    node[right,black,font=\scriptsize] {$u$};
  \draw[guide] ($(D)+(0,-1.32)$)--($(D)+(0,1.32)$)
    node[above,black,font=\scriptsize] {$v$};
  \draw[guide,densely dashed]
    ($(D)+(28:1.18)$)--($(D)+(208:1.18)$);
  \coordinate (z) at ($(D)+(28:0.84)$);
  \coordinate (mz) at ($(D)+(208:0.84)$);
  \draw[guide] (D)--(z) node[midway,above,font=\scriptsize,black] {$r$};
  \draw[mapaccent,-{Latex[length=1.5mm,width=1.1mm]},line width=0.55pt]
    ($(D)+(0:0.36)$) arc[start angle=0,end angle=28,radius=0.36];
  \node[font=\scriptsize] at ($(D)+(14:0.57)$) {$\theta$};
  \node[point,label={[font=\scriptsize]above right:$z$}] at (z) {};
  \node[point,label={[font=\scriptsize]below left:$-z$}] at (mz) {};
  \node[panel title] at ($(D)+(0,1.62)$) {Unit disk $\B^2$};

  \draw[map line,->] (-1.65,0.18)--(0.15,0.18);
  \node[font=\small,above] at (-0.75,0.18) {$Q$};
  \node[font=\scriptsize,align=center,below=3pt] at (-0.75,0.18)
    {$t=r^2$\,,\quad $\phi=2\theta$};

  \coordinate (A) at (2.65,-1.25);
  \coordinate (L) at (1.25,0.95);
  \coordinate (R) at (4.05,0.95);
  \path[fill=mapfill] (A)--(L)--(R)--cycle;
  \draw[guide,densely dashed] (A)--(2.65,1.43);
  \draw[map line] (L)--(A)--(R);
  \draw[map line,densely dashed] (R)
    arc[start angle=0,end angle=180,x radius=1.40,y radius=0.30];
  \draw[map line] (L)
    arc[start angle=180,end angle=360,x radius=1.40,y radius=0.30];

  \draw[guide,densely dashed] (3.50,0.09)
    arc[start angle=0,end angle=180,x radius=0.85,y radius=0.18];
  \draw[guide] (1.80,0.09)
    arc[start angle=180,end angle=360,x radius=0.85,y radius=0.18];
  \coordinate (qz) at (3.13,-0.06);
  \node[point] at (qz) {};
  \draw[guide] ($(qz)+(0.08,0.03)$)--(3.63,-0.22);
  \node[font=\scriptsize,anchor=west,align=left] at (3.67,-0.25)
    {$Q(z)=Q(-z)$};
  \node[font=\scriptsize,anchor=west] at (4.10,0.96) {$t=1$};
  \node[font=\scriptsize,anchor=north] at ($(A)+(0,-0.07)$) {$t=0$};
  \node[panel title] at (2.65,1.65) {Conic surface $V$};
\end{tikzpicture}
\caption{The quadratic quotient $Q:\B^2\to V$.  If
$z=(u,v)=(r\cos\theta,r\sin\theta)$, then
$Q(z)=(r^2\cos 2\theta,r^2\sin 2\theta,r^2)$.  Thus $Q$ identifies
$z$ and $-z$, is two-to-one away from the origin, and maps the origin to
the apex.}
\label{fig:quadratic-quotient}
\end{figure}

{\color{black}
Quadratic changes of variables have long been used to transfer orthogonal polynomials and cubature among balls, simplices, and spheres; see Xu~\cite{Xu01JCAM, Xu06ETNA}.  The disk-to-cone transformation itself is likewise classical.  For the surface weight $t^\beta(1-t)^\gamma\,\dd\sigma$, the present probability measure corresponds to $\beta=-1$, and Xu's conic orthogonal polynomials admit explicit quadratic transformations to complex disk polynomials~\cite{XuCone2020}.  The framework in~\cite{XuSymmetric2026} further emphasizes quadratic maps between symmetric domains.  We therefore claim no novelty for $Q$ or for its action on individual basis functions.
}


{\color{black}Our contribution begins with the use of $Q$ at the level of \emph{full trace spaces, positive cubature, sampling matrices, reproducing kernels, and hyperinterpolation operators}.  This yields the cone-specific two-way correspondence for positive rules, its exact node-count consequences, and the rule-independent stability bounds.}
{\color{black}For a function $f$ on $V$, write $Q^*f=f\circ Q$. The subscript “even” denotes invariance under $z\mapsto-z$; thus $L^2_{\mathrm{even}}(\mathbb B^2,\mu_\gamma)$ is the even subspace of $L^2(\mathbb B^2,\mu_\gamma)$, and
$$
\Pi_{2m}^{\mathrm{even}}(\mathbb B^2)
:=\{P\in\Pi_{2m}(\mathbb B^2):P(-z)=P(z)\}.
$$}For the disk probability measure
\begin{equation}\label{eq:mu}
\dd\mu_\gamma(u,v)=\frac{\gamma+1}{\pi}(1-u^2-v^2)^\gamma\,\dd u\,\dd v,
\end{equation}
we prove
\begin{equation}\label{eq:quotient-intro}
Q^*:L^2(V,\nu_\gamma)\xrightarrow{\ \cong\ }L^2_{\even}(\B^2,\mu_\gamma),
\qquad Q^*\Pi_m(V)=\Pi^{\even}_{2m}(\B^2).
\end{equation}
This basis-free quotient turns the cone problem into a symmetric disk problem without losing positivity or exactness.

The main conclusions can be summarized as follows.

{\color{black}We write $L^V_{n,\gamma}$ for the degree-$n$ cone
hyperinterpolation operator,
$$
\Lambda^V_{n,\gamma}
:=\|L^V_{n,\gamma}\|_{C(V)\to C(V)}
$$
for its Lebesgue constant, and $H_m^B$ for the degree-$m$
disk hyperinterpolation operator.}

\begin{theorem}(Compression and stability principle)\label{thm:main-summary}
Let $\gamma>-1$ and let a positive cubature rule on $(V,\nu_\gamma)$ be exact on $\Pi_{2n}(V)$.
\begin{enumerate}[label=\textnormal{(\roman*)}]
\item It lifts bijectively to a centrally symmetric positive disk rule of degree $4n+1$.  Under this lifting, M\"oller's disk lower bound is exactly the cone dimension bound $N\ge D_n$, and the excesses above the two bounds are related by an explicit  {\color{black} even–odd node-count identity}.
\item {\color{black} In an $L^2(\nu_\gamma)$-orthonormal basis, the weight-scaled sampling matrix on $\Pi_n(V)$ has orthonormal columns; hence its $2$-norm condition number is one in exact arithmetic.}
\item{\color{black}The associated degree-$n$ hyperinterpolation operator satisfies
$$
\Lambda_{n,\gamma}^V \ge c_\gamma n^{\gamma+1}.
$$
If $\gamma\in\tfrac12\mathbb Z_{\ge0}$, the matching upper bound holds and therefore
$$
      c_\gamma n^{\gamma+1}\le \Lambda_{n,\gamma}^V
      \le C_\gamma n^{\gamma+1},
$$
with constants independent of both $n$ and the particular positive cubature rule.}
\item {\color{black} For normalized area measure on the disk, every positive rule with nodes in the closed disk and exact through degree $2m$ generates a disk hyperinterpolation operator satisfying $\|H_m^B\|\ge cm$, without any central-symmetry assumption.}
\end{enumerate}
\end{theorem}

The upper estimate in part~(iii) follows by transferring Wade's disk estimate to the even subspace.  The lower estimate requires a separate argument: a rim average removes all assumptions on the angular arrangement of the nodes, while positivity and exactness force a fixed amount of cubature mass into an interior radial annulus.  This produces a rule-independent $L^1$ lower bound for the sampled reproducing kernel.  The key point is that the lower estimate cannot be obtained merely from a lower bound for the full disk operator, because the conic operator corresponds to the restriction to even data.

The resulting computational picture is therefore
$$
\boxed{
\text{fewer samples}
\quad+\quad
\text{exact }L^2\text{ conditioning}
\quad+\quad
\text{optimal universal }L^\infty\text{ stability}.
}
$$
For $\gamma=0$, known near-minimal disk formulas give rigorous cone support bounds $4,10,18,29$ for $n=1,2,3,4$, compared with $6,15,28,45$ product nodes.  Higher-degree optimized disk schemes provide further numerical examples.

The paper is organized to follow this computational chain. {\color{black} \Cref{sec:trace} gives the product benchmark and the dimension lower bound.  \Cref{sec:quotient} establishes the measure and polynomial quotient.  \Cref{sec:cubature} proves the cubature correspondence, records transferred rules, and states the scope of the exact $L^2$ conditioning result.  \Cref{sec:operator} transfers kernels and hyperinterpolation.  \Cref{sec:Lebesgue} proves the weighted cone lower bound, the sharp two-sided law in the half-integer range, and the arbitrary-rule disk lower bound.  }\Cref{sec:prony} records the dimension-matched Prony scale for the all-degree construction problem.

\section{Conic trace spaces and the product benchmark}\label{sec:trace}

Throughout, a positive cubature rule means a rule with distinct nodes and strictly positive weights; zero-weight nodes are discarded.  Exactness for real polynomials is extended complex-linearly when complex notation is used.

Let $\Pi_m(\R^3)$ denote {\color{black}the space of polynomials}  of total degree at most $m$, and define the trace space
$$
\Pi_m(V):=\{P|_V:P\in\Pi_m(\R^3)\}.
$$
The relation $x^2+y^2=t^2$ yields the Fourier--radial description
\begin{equation}\label{eq:fourier-radial}
\PiC_m(V)=\operatorname{span}\{t^\ell e^{ik\phi}:0\le\ell\le m,\ |k|\le\ell\}.
\end{equation}
The functions in \eqref{eq:fourier-radial} are linearly independent: distinct Fourier modes are independent in $\phi$, and for each fixed $k$ the remaining powers of $t$ are independent on $(0,1)$.  Hence
\begin{equation}\label{eq:dimension}
\dim\Pi_m(V)=\sum_{\ell=0}^m(2\ell+1)=(m+1)^2.
\end{equation}

The following product rule is the Gaussian-in-radius/equispaced-in-angle specialization of the standard cubature construction on quadratic surfaces; compare~\cite{OlverXu2020}.

\begin{proposition}\label{prop:product}
Let $s\ge0$,
$$
m_s=\left\lfloor\frac{s}{2}\right\rfloor+1,
\qquad L_s=s+1.
$$
Let $\{(t_k,A_k)\}_{k=1}^{m_s}$ be the positive Gauss--Jacobi rule for the probability measure $(\gamma+1)(1-t)^\gamma\dd t$ on $[0,1]$, and let $\phi_j=2\pi j/L_s$.  Then
\begin{equation}\label{eq:product-rule}
Q_s f=\sum_{k=1}^{m_s}\frac{A_k}{L_s}\sum_{j=0}^{L_s-1}
 f(t_k\cos\phi_j,t_k\sin\phi_j,t_k)
\end{equation}
is positive and exact on $\Pi_s(V)$.  Its algebraic degree is exactly $s$, and it uses
\begin{equation}\label{eq:product-count}
N_s=(s+1)\left(\left\lfloor\frac{s}{2}\right\rfloor+1\right)
\end{equation}
nodes.
\end{proposition}

\begin{proof}
It suffices to use the spanning functions in \eqref{eq:fourier-radial}.  For angular frequency  $k=0$, Gauss--Jacobi exactness applies because $\ell\le s\le2m_s-1$.  For $0<\abs{k}\le\ell\le s$, both the continuous angular integral and the discrete angular average vanish, since $\abs{k}<L_s$.  Thus $Q_s$ is exact on $\Pi_s(V)$.

To see that the degree cannot be larger, take
$$
q_{s+1}(x,y,t)=\operatorname{Re}(x+iy)^{s+1}=t^{s+1}\cos((s+1)\phi).
$$
Its exact integral is zero, whereas $\cos((s+1)\phi_j)=1$ at every angular node.  Hence
$$
Q_s q_{s+1}=\sum_k A_k t_k^{s+1}>0.
$$
\end{proof}

For $s=2n$, \Cref{prop:product} gives $(n+1)(2n+1)$ nodes.  Its exactness implies
\begin{equation}\label{eq:product-MZ}
\sum_i\omega_i\abs{p(\xi_i)}^2=\norm{p}_{L^2(\nu_\gamma)}^2,
\qquad p\in\PiC_n(V).
\end{equation}
Thus the weighted Vandermonde matrix has orthonormal columns.

The universal lower bound is equally elementary but computationally decisive.

\begin{proposition}\label{prop:dimension-lower}
Every positive cubature rule on $(V,\nu_\gamma)$ that is exact on $\Pi_{2n}(V)$ satisfies
$$
N\ge D_n=(n+1)^2.
$$
\end{proposition}

\begin{proof}
If $N<\dim\Pi_n(V)$, there is a nonzero $p\in\Pi_n(V)$ vanishing at all nodes.  Exactness applied to $\abs p^2$ gives $\norm p_{L^2(\nu_\gamma)}=0$.  Since $\nu_\gamma$ has full support on the relative interior of $V$ and $p$ is continuous, the trace polynomial must vanish identically on $V$, a contradiction.
\end{proof}

Consequently,
$$
\frac{N_n^{\pr}}{D_n}=2-\frac1{n+1}.
$$
The next section identifies the symmetric disk problem behind this factor-of-two gap.

\section{The quadratic map as a measure-preserving quotient}\label{sec:quotient}

Let $\B^2=\{(u,v)\in\R^2:u^2+v^2\le1\}$ and let $\mu_\gamma$ be given by \eqref{eq:mu}. {\color{black}We denote by $\Pi_m(\B^2)$ the restrictions to
$\B^2$ of bivariate polynomials of total degree at most $m$.}

\subsection{Measure identity}
{\color{black}The measure-theoretic compatibility of $Q$ with the disk
and cone weights is summarized by the following pushforward identity.}

\begin{theorem}\label{thm:measure-lifting}
For every integrable $f:V\to\C$,
\begin{equation}\label{eq:measure-lifting}
\int_V f(x,y,t)\,\dd\nu_\gamma
=
\int_{\B^2}f(Q(u,v))\,\dd\mu_\gamma(u,v).
\end{equation}
Equivalently, $Q_\#\mu_\gamma=\nu_\gamma$.
\end{theorem}

\begin{proof}
Write $u=r\cos\theta$, $v=r\sin\theta$.  Then
$$
Q(r,\theta)=(r^2\cos2\theta,r^2\sin2\theta,r^2),
$$
so $t=r^2$, $\phi=2\theta$, and $r\dd r=\dd t/2$.  Hence
\begin{align*}
\int_{\B^2}f(Q(u,v))\,\dd\mu_\gamma
&=\frac{\gamma+1}{\pi}\int_0^1\int_0^{2\pi}
 f(r^2\cos2\theta,r^2\sin2\theta,r^2)(1-r^2)^\gamma r\,\dd\theta\,\dd r\\
&=\frac{\gamma+1}{2\pi}\int_0^1(1-t)^\gamma
 \int_0^{2\pi}f(t\cos2\theta,t\sin2\theta,t)\,\dd\theta\,\dd t.
\end{align*}
The inner integrand is $\pi$-periodic in $\theta$, and the change of variable $\phi=2\theta$ gives \eqref{eq:measure-lifting}.
\end{proof}

{\color{black}This measure identity also identifies conic functions isometrically with even disk
functions.}

\begin{corollary}\label{cor:even-isometry}
The pullback $Q^*f=f\circ Q$ defines an
isometric isomorphism 
$$
L^2(V,\nu_\gamma)\cong L^2_{\even}(\B^2,\mu_\gamma),
\qquad
L^2_{\even}(\B^2,\mu_\gamma)=\{F:F(-z)=F(z) {\color{black} \ \ a.e.}\}.
$$
{\color{black}It also defines an isometric isomorphism  $C(V)\cong C_{\even}(\B^2)$ with respect to the
uniform norm.}
\end{corollary}

\begin{proof}
The isometry follows from \Cref{thm:measure-lifting}.  Since $Q(z)=Q(-z)$, the range is even.  Conversely, every nonzero fiber of $Q$ is exactly $\{z,-z\}$ and the origin has the single preimage $0$, so an even measurable function is constant on fibers and factors through $Q$ up to the usual almost-everywhere identification.  For continuous functions, $Q:\B^2\to V$ is a continuous surjection from a compact space onto a Hausdorff space and hence a quotient map.  The uniform norm is preserved.
\end{proof}

\subsection{Full polynomial trace spaces}

The basis-level quadratic transformation is classical in the present parameter regime~\cite{XuCone2020}.  The discrete theory requires the following statement for the entire trace spaces.

\begin{theorem}\label{thm:poly-quotient}
For every integer $m\ge0$, pullback by $Q$ is a linear isomorphism
\begin{equation}\label{eq:poly-quotient}
Q^*:\Pi_m(V)\longrightarrow\Pi^{\even}_{2m}(\B^2),
\qquad
\Pi^{\even}_{2m}(\B^2)=\{P\in\Pi_{2m}(\B^2):P(-z)=P(z)\}.
\end{equation}
Moreover, the map is unitary for the corresponding $L^2$ inner products.
\end{theorem}

\begin{proof}
Every polynomial in $x,y,t$ of degree at most $m$ becomes, after substitution \eqref{eq:Q}, an even polynomial in $u,v$ of degree at most $2m$.  Hence
$Q^*\Pi_m(V)\subseteq\Pi^{\even}_{2m}(\B^2)$.

Conversely, the invariant ring of $(u,v)\mapsto(-u,-v)$ is generated by $u^2,uv,v^2$, and
\begin{equation}\label{eq:quadratic-generators}
u^2=\frac{t+x}{2},\qquad v^2=\frac{t-x}{2},\qquad uv=\frac y2.
\end{equation}
Decompose an even polynomial into homogeneous pieces.  A homogeneous piece of degree $2j$ is a polynomial of degree $j$ in $u^2,uv,v^2$, hence is the pullback of a polynomial in $x,y,t$ of degree $j$.  The relation
$$
(u^2-v^2)^2+(2uv)^2=(u^2+v^2)^2
$$
is precisely the defining relation of $V$, so the representation is well defined on the quotient.  Injectivity follows from $Q(\B^2)=V$, and unitarity follows from \Cref{cor:even-isometry}.
\end{proof}

\begin{remark}\label{rem:classical-transform}
No new basis transformation is asserted in \Cref{thm:poly-quotient}.  Its role is to promote the classical quadratic identity to a statement about full trace spaces and then combine it with positivity, symmetry, and exactness.  The genuinely discrete consequences begin in \Cref{sec:cubature}.
\end{remark}

The quotient also explains the dimension identity
$$
\dim\Pi^{\even}_{2m}(\B^2)=\sum_{j=0}^m(2j+1)=(m+1)^2=\dim\Pi_m(V).
$$
Thus conic degree $m$ corresponds exactly to even disk degree $2m$.

\section{Positive cubature compression and exact conditioning}\label{sec:cubature}

\subsection{Exact cone--disk correspondence}

A disk cubature rule is called centrally symmetric if every nonzero node $z$ occurs together with $-z$ with the same weight.

{\color{black}At the discrete level, $Q$ transfers positive cone rules to centrally symmetric
positive disk rules and pushes arbitrary positive disk rules back to the cone.}

\begin{theorem}\label{thm:cubature-equivalence}
Fix $n\ge0$.
\begin{enumerate}[label=\textnormal{(\roman*)}]
\item Let $\sum_{j=1}^N\lambda_j f(\xi_j)$ be a positive cubature formula on $(V,\nu_\gamma)$ exact on $\Pi_{2n}(V)$.  Choose one preimage $z_j\in\B^2$ with $Q(z_j)=\xi_j$ for every non-apex node.  Replace each such node by the pair $\pm z_j$, each with weight $\lambda_j/2$; if the apex occurs, retain the origin with its original weight.  The resulting disk formula is positive, centrally symmetric, and exact on $\Pi_{4n+1}(\B^2)$.
\item {\color{black} More generally, pushing any positive disk cubature formula with nodes in the closed disk and exact on
$\Pi_{4n}(\B^2)$ through $Q$, and merging nodes with the same image, gives a positive cone
cubature formula exact on $\Pi_{2n}(V)$.}
\item[(iii)] {\color{black}
Restricted to centrally symmetric positive disk rules of degree $4n+1$, the
two constructions are inverse; hence the lifting in~(i) is a bijection.
}
\end{enumerate}
\end{theorem}

\begin{proof}
For (i), let $P\in\Pi_{4n+1}(\B^2)$ and write
$$
P=P_++P_-,\qquad P_\pm(z)=\tfrac12(P(z)\pm P(-z)).
$$
The symmetric discrete rule and the radial measure $\mu_\gamma$ annihilate $P_-$.  The even part has degree at most $4n$, so by \Cref{thm:poly-quotient} there exists $p\in\Pi_{2n}(V)$ such that $P_+=p\circ Q$.  The disk sum of $P_+$ equals the cone sum of $p$, which equals its cone integral; \Cref{thm:measure-lifting} identifies this with the disk integral of $P_+$.  Hence the lifted rule is exact through degree $4n+1$.

{\color{black}The lifted nodes are distinct: a non-apex preimage satisfies $z_j\ne0$, hence
$z_j\ne-z_j$; moreover, if $\xi_j\ne\xi_k$, then $z_j\ne\pm z_k$, because
$Q(\pm z_k)=\xi_k$.}

{\color{black}For~(ii), if $p\in\Pi_{2n}(V)$ then $p\circ Q\in\Pi_{4n}^\even(\B^2)$. Disk exactness
and Theorem~\ref{thm:measure-lifting} give the required cone exactness. Positivity is preserved
when coincident images are merged by adding their weights.  }{\color{black}Finally, central symmetry
fixes both members and the equal weights of each nonzero fiber $\{z,-z\}$, so~(iii) follows.}
\end{proof}
{\color{black}Write $M_\disk$ and $N_\cone$ for the disk and cone node counts, respectively;
the correspondence fixes their relation.}
\begin{corollary}\label{cor:node-count}
If the lifted disk rule contains the origin, then
$$
M_{\disk}=2N_{\cone}-1,
\qquad
N_{\cone}=\frac{M_{\disk}+1}{2}.
$$
If it does not contain the origin, then $M_{\disk}=2N_{\cone}$.
\end{corollary}

\subsection{M\"oller's lower bound and the exact excess identity}
{\color{black}This subsection recalls the classical lower-bound argument needed to interpret the transferred node counts. The novelty claimed here is the cone–disk bijection and its exact excess relation, not M\"oller's bound itself.}

For a centrally symmetric measure in two variables, M\"oller's lower bound for a degree-$(2k-1)$ cubature rule is~\cite{Moller1976}
\begin{equation}\label{eq:moller}
M\ge\dim\Pi_{k-1}(\R^2)+\left\lfloor\frac{k}{2}\right\rfloor.
\end{equation}

\begin{theorem}[M\"oller quotient bound]\label{thm:moller-quotient}
Every positive degree-$2n$ cubature formula on $(V,\nu_\gamma)$ satisfies
$$
N\ge(n+1)^2.
$$
Under the quadratic lifting, this is exactly M\"oller's lower bound for centrally symmetric degree-$(4n+1)$ disk cubature.
\end{theorem}

\begin{proof}
Set $2k-1=4n+1$, so $k=2n+1$.  Then \eqref{eq:moller} gives
$$
M_{\disk}\ge \dim\Pi_{2n}(\R^2)+n
=\frac{(2n+1)(2n+2)}2+n
=2(n+1)^2-1.
$$
If the origin is present, \Cref{cor:node-count} gives $N_{\cone}\ge(n+1)^2$.  If it is absent, $M_{\disk}=2N_{\cone}$ is even, whereas the lower bound is odd, so $M_{\disk}\ge2(n+1)^2$ and again $N_{\cone}\ge(n+1)^2$.
\end{proof}
{\color{black}The coincidence of the two lower bounds yields an exact relation between the
sampling excesses.}
\begin{corollary}\label{cor:excess}
Let
$$
D_n=(n+1)^2,
\qquad
M_n^{\mathrm M}=2D_n-1.
$$
{\color{black}Here $M_n^{\mathrm M}$ denotes M\"oller's disk lower bound at degree $4n+1$.} For a lifted centrally symmetric disk rule,
\begin{equation}\label{eq:excess}
M_{\disk}-M_n^{\mathrm M}
=2(N_{\cone}-D_n)+\delta,
\qquad 
\delta=
\begin{cases}
0,&0\text{ is a disk node},\\
1,&0\text{ is not a disk node}.
\end{cases}
\end{equation}
Thus the disk excess is exactly twice the cone excess when the origin is present, and differs from twice the cone excess only by the unavoidable parity offset otherwise.
\end{corollary}

\subsection{Rigorous low-degree compression for $\gamma=0$}

For $\gamma=0$, the catalogue summarized and rederived by Benouahmane, Cuyt and Yaman~\cite{BenouahmaneCuytYaman2019} contains centrally symmetric near-minimal rules of degrees $5,9,13,17$ with $7,19,35,57$ nodes, respectively, with nodes in the closed disk and nonnegative weights.  Pushing them through $Q$ gives \Cref{tab:rigorous-compression}; zero-weight nodes in a tabulated representation are discarded.

{\color{black}The tabulated disk constructions therefore yield the following low-degree non-product cone rules.}
\begin{corollary}\label{cor:low-degree}
For $\gamma=0$, the disk rules listed in~\cite{BenouahmaneCuytYaman2019} push forward to non-product degree-$2n$ cone formulas with nonnegative weights and the support sizes in \Cref{tab:rigorous-compression}.  If a representation contains a zero-weight node, deleting it gives a positive rule with at most the displayed number of nodes.
\end{corollary}

{\color{black}For a fixed transferred rule
$\{(\xi_j,\lambda_j)\}$, define its Lebesgue function by
$$
\lambda_n^V(\xi)
   :=\sum_j\lambda_j\bigl|K_n^V(\xi,\xi_j)\bigr|.
$$
We first evaluate $\lambda_n^V$ on the base height--angle
grid $\mathcal G_n$ and then locally refine the leading grid
candidates.  The final column below reports the largest value
found after this refinement, denoted by
$\widehat\Lambda_n^V$.  The grids $\mathcal G_n$ and the
refinement procedure are described in the numerical subsection
below.  The hat indicates that these are numerical estimates,
not certified global suprema.}

\begin{table}[ht]
\centering
\caption{Rigorous compression obtained from centrally symmetric disk formulas {\color{black}listed in~\cite[Table~1, p.~300]{BenouahmaneCuytYaman2019}}.  {\color{black}The two columns preceding the final one display the sampling excess above the information-theoretic dimension $D_n$; the final column gives a numerical Lebesgue-constant estimate for the indicated representative.}}
\label{tab:rigorous-compression}
\scriptsize
\setlength{\tabcolsep}{2.7pt}
\begin{tabular}{c c c c c c c c c}
\toprule
$n$ & cone deg. & $D_n$ & disk deg. & disk nodes & product nodes & cone nodes & $(N_n^{\pr}-D_n,\,N-D_n)$ &  {\color{black}$\widehat\Lambda_n^V$}\\
\midrule
1 & 2 & 4  & 5  & 7  & 6  & 4  & $(2,0)$  & {\color{black}$3.000$}\\
2 & 4 & 9  & 9  & 19 & 15 & 10 & $(6,1)$ & {\color{black}$4.824$}\\
3 & 6 & 16 & 13 & 35 & 28 & 18 & $(12,2)$ & {\color{black}$6.666$}\\
4 & 8 & 25 & 17 & 57 & 45 & 29 & $(20,4)$ & {\color{black}$7.868$}\\
\bottomrule
\end{tabular}
\end{table}

{\color{black} The $n=2$ and $n=4$ Lebesgue-constant estimates in \Cref{tab:rigorous-compression} use, respectively, Albrecht's $19$-node degree-$9$ rule~\cite{Albrecht1960} and the $57$-node degree-$17$ rule of Cools and Kim~\cite{CoolsKim2000}.  The degree-$13$ entry uses our numerically reconstructed positive $6R$--$3A_x$--$2A_y$--$O$ representative. Here $R$ denotes a four-point reflection orbit, $A_x$ and $A_y$ denote centrally symmetric pairs on the two coordinate axes, and $O$ denotes the origin; see Appendix~\ref{app:degree13-data}.}
{\color{black} The degree-$13$ moment system in~\cite[Section~3]{BenouahmaneCuytYaman2019} is nonunique.  We selected one representative by initializing from the $36$-node fully symmetric rule listed in the same table, collapsing one axis pair to the origin, and solving the $28$ normalized even-moment equations with positivity and disk constraints.  Its smallest disk weight is $5.82\times10^{-3}$ and its largest radius is $0.997833$.}
{\color{black} The largest absolute disk-moment residual through total degree $13$ is $2.8\times10^{-17}$ for the full-precision computed solution and remains below $1.4\times10^{-16}$ after the orbit parameters are rounded to the $17$ significant digits displayed in \Cref{tab:degree13-orbits}.  The parameters in that table, located in Appendix~\ref{app:degree13-data}, determine all $35$ disk nodes and weights.}
{\color{black}The value in the table is not asserted for every member of the degree-$13$ family.}

For $n=1$, the four-point rule consists of the apex and an equilateral triangle on the section $t=2/3$, all with equal weights.  It is the pushforward of the seven-point centrally symmetric degree-five disk rule.

The excess columns in \Cref{tab:rigorous-compression} make the computational gain explicit: the tensor product rule carries an excess that grows rapidly even at low degree, while the transferred non-product rules lie close to the dimension limit.

\subsection{Transfer of optimized disk schemes}

The quotient also transfers numerically optimized disk schemes.  Takaki, Forbes and Rolland~\cite{TakakiForbesRolland2022} report positive-weight disk configurations obtained by solving moment equations numerically.  Their displayed centrally symmetric schemes include degrees $17,21,25,29$ with $55,85,117,155$ points, respectively.  Since these cardinalities are odd, the displayed centrally symmetric configurations contain the origin, and their quotient support counts are $(M+1)/2$.

\begin{table}[t]
\centering
\caption{Compression obtained from the positive-weight disk schemes reported in~\cite{TakakiForbesRolland2022}.  The source schemes are double-precision numerical formulas; the transferred cone counts have the same numerical status.  {\color{black}The final column gives sampled and locally refined Lebesgue-constant estimates.}}
\label{tab:numerical-compression}
\small
\setlength{\tabcolsep}{3.8pt}
\begin{tabular}{c c c c c c c c}
\toprule
$n$ & cone deg. & $D_n$ & product nodes & disk deg. & disk nodes & cone nodes & {\color{black}$\widehat\Lambda_n^V$}\\
\midrule
4 & 8  & 25 & 45  & 17 & 55  & 28 & {\color{black}$8.549$}\\
5 & 10 & 36 & 66  & 21 & 85  & 43 & {\color{black}$9.932$}\\
6 & 12 & 49 & 91  & 25 & 117 & 59 & {\color{black} $12.776$}\\
7 & 14 & 64 & 120 & 29 & 155 & 78 & {\color{black} $12.773$}\\
\bottomrule
\end{tabular}
\end{table}

{\color{black}These estimates use the supplementary representatives \texttt{T17\_6segment}, \texttt{T21\_6segment}, \texttt{T25\_FullSym}, and \texttt{T29\_FullSym} from~\cite{TakakiForbesRolland2022}.  Since the optimized schemes need not be unique, the numerical values are attached to those representatives, not merely to their degrees and node counts.}

The implication from disk moments to cone moments is exact; only the supplied disk data are numerical.  Thus \Cref{tab:numerical-compression} should be read with the same floating-point certification level as the source schemes.  In particular, these data are not used in any proof below.

\subsection{Compression preserves exact $L^2$ conditioning}

The next statement is elementary but central for the computational interpretation: reducing the node count does not degrade the exact $L^2$ sampling condition number.

{\color{black}Let $K_n^V$ denote the reproducing kernel of $\Pi_n^\mathbb{C}(V)$ in $L^2(\nu_\gamma)$. We use the convention
$$
\langle f,g\rangle = \int_V f\overline{g}\,d\nu_\gamma,\qquad p(\xi) = \langle p, K_n^V(\cdot,\xi)\rangle.
$$
Thus, for any orthonormal basis $\{P_\ell\}_{\ell=1}^{D_n}$,
$$
K_n^V(\eta,\xi) = \sum_{\ell=1}^{D_n} P_\ell(\eta)\overline{P_\ell(\xi)}.
$$

The following proposition records the exact discrete $L^2$ identity and its conditioning consequences.}
\begin{proposition}\label{prop:MZ-conditioning}
Let $\{(\xi_j,\lambda_j)\}_{j=1}^N$ be any positive cone cubature formula exact on $\Pi_{2n}(V)$.  Then for every $p\in\PiC_n(V)$,
\begin{equation}\label{eq:MZ}
\sum_{j=1}^N\lambda_j\abs{p(\xi_j)}^2=\norm p_{L^2(\nu_\gamma)}^2.
\end{equation}
Let $\{P_\ell\}_{\ell=1}^{D_n}$ be any orthonormal basis of $\PiC_n(V)$ and define
$$
\mathrm{S}_{j\ell}=\sqrt{\lambda_j}\,P_\ell(\xi_j).
$$
Then
\begin{equation}\label{eq:Uisometry}
\mathrm{S}^*\mathrm{S}=I_{D_n},
\qquad
\kappa_2(\mathrm{S})=1.
\end{equation}{\color{black}
For a sample vector $a=(a_j)$, set}
$$
\norm a_{\ell^2_\lambda}=\left(\sum_j\lambda_j\abs{a_j}^2\right)^{1/2},
\qquad
Ta=\sum_j\lambda_j a_jK_n^V(\,\cdot\,,\xi_j),
$$
then
$$
\norm T_{\ell^2_\lambda\to L^2(\nu_\gamma)}=1.
$$
\end{proposition}

\begin{proof}
Since $\abs p^2=p\overline p$ has degree at most $2n$, exactness gives \eqref{eq:MZ}.  Applying it to pairs of basis functions gives $\mathrm{S}^*\mathrm{S}=I_{D_n}$.  Thus all singular values of $\mathrm{S}$ are one.

For a sample vector $a$, set $\widetilde a_j=\sqrt{\lambda_j}a_j$.  The coefficient vector of $Ta$ in the orthonormal basis is $\mathrm{S}^*\widetilde a$, hence
$$
\norm{Ta}_{L^2}=\norm{\mathrm{S}^*\widetilde a}_2\le\norm{\widetilde a}_2=\norm a_{\ell^2_\lambda}.
$$
Equality is attained {\color{black}for every nonzero} $\widetilde a\in\ran \mathrm{S}$.
\end{proof}

\begin{remark}[A practical certification diagnostic]\label{rem:diagnostic}
For an exact rule, \eqref{eq:Uisometry} is an identity.  For a numerically optimized candidate, the quantities
$$
\norm{\mathrm{S}^*\mathrm{S}-I}_{2}
\quad\text{and}\quad
\max_{\abs\alpha\le2n}\left|\sum_j\lambda_j\xi_j^\alpha-\int_V\xi^\alpha\,\dd\nu_\gamma\right|
$$
provide natural conditioning and moment-residual diagnostics.  This distinction is important for \Cref{tab:numerical-compression}, whose source data are numerical rather than certified exact formulas.
\end{remark}

\section{Reproducing kernels and hyperinterpolation under the quotient}\label{sec:operator}

Let $K_n^V$ be the reproducing kernel of $\Pi_n(V)$ in $L^2(\nu_\gamma)$, and let $K_m^B$ be the reproducing kernel of $\Pi_m(\B^2)$ in $L^2(\mu_\gamma)$.  Since the disk weight is radial,  {\color{black}the
even–odd decomposition is }orthogonal.  Define
\begin{equation}\label{eq:even-kernel}
K_{2n}^{B,+}(z,w)=\frac12\bigl(K_{2n}^B(z,w)+K_{2n}^B(z,-w)\bigr).
\end{equation}
{\color{black}The following theorem identifies the conic reproducing kernel with the even
component of the lifted disk kernel.}
\begin{theorem}\label{thm:kernel}
For all $z,w\in\B^2$,
\begin{equation}\label{eq:kernel-identity}
K_n^V(Qz,Qw)=K_{2n}^{B,+}(z,w).
\end{equation}
\end{theorem}

\begin{proof}
Choose an orthonormal basis $\{P_j\}_{j=1}^{(n+1)^2}$ of $\Pi_n(V)$.  By \Cref{thm:poly-quotient}, $\{P_j\circ Q\}$ is an orthonormal basis of $\Pi_{2n}^{\even}(\B^2)$.  Summing the basis expansion of the reproducing kernel gives \eqref{eq:kernel-identity}; \eqref{eq:even-kernel} is the orthogonal {\color{black}even-subspace} projection of the full disk kernel.
\end{proof}

Let a positive cone cubature rule of degree $2n$ be given and lift it by \Cref{thm:cubature-equivalence}.  Define conic hyperinterpolation by
\begin{equation}\label{eq:hyperinterp}
L_n^Vf(\xi)=\sum_{j=1}^N\lambda_jf(\xi_j)K_n^V(\xi,\xi_j).
\end{equation}
{\color{black}Degree-$2n$ exactness and the reproducing property imply $L_n^V p = p$ for every $p \in \Pi_n(V)$; thus $L_n^V$ is a projection onto $\Pi_n(V)$.}
Let $H_{2n}^B$ denote the ordinary degree-$2n$ disk hyperinterpolation generated by the lifted disk rule. {\color{black} If the lifted nodes and weights are $(z_i,w_i)$, then
$$
 H_{2n}^BF(z)=\sum_iw_iF(z_i)K_{2n}^B(z,z_i).
$$}

{\color{black}Conic hyperinterpolation therefore agrees with disk hyperinterpolation restricted
to even data.}
\begin{theorem}\label{thm:operator-transfer}
If $F=f\circ Q$, then
\begin{equation}\label{eq:operator-transfer}
(L_n^Vf)\circ Q=H_{2n}^BF.
\end{equation}
The right-hand side is even.  Consequently,
\begin{equation}\label{eq:norm-transfer}
\norm{L_n^V}_{C(V)\to C(V)}
={\color{black}\norm{H_{2n}^B|_{C_{\even}(\B^2)}}_{C_{\even}(\B^2)\to C(\B^2)}}
\le\norm{H_{2n}^B}_{C(\B^2)\to C(\B^2)}.
\end{equation}
\end{theorem}

\begin{proof}
The lifted pairs carry equal weights.  For even data $F$, every odd disk basis function has zero discrete coefficient by pairwise cancellation, so the full disk hyperinterpolant equals its even part.  Combining \Cref{thm:kernel} with pairwise merging of $\pm z_j$ gives \eqref{eq:operator-transfer}.  The norm identity follows from the isometry $C(V)\cong C_{\even}(\B^2)$ in \Cref{cor:even-isometry}.
\end{proof}

\section{Optimal $L^\infty$ stability: {\color{black}Rule-independent bounds}}\label{sec:Lebesgue}
{\color{black}
For a cone rule of degree $2n$, define
\begin{equation}\label{eq:Lebesgue-function}
\lambda_{n,\gamma}^V(\xi)=\sum_j\lambda_j\abs{K_{n,\gamma}^V(\xi,\xi_j)},
\qquad
\Lambda_{n,\gamma}^V=\sup_{\xi\in V}\lambda_{n,\gamma}^V(\xi).
\end{equation}
At $\gamma=0$ we suppress the parameter in the notation.  For distinct nodes, $\Lambda_{n,\gamma}^V$ is exactly the $C(V)\to C(V)$ operator norm: at a fixed evaluation point, signs (or complex phases) of the nodal data may be prescribed independently and extended continuously to $V$.

When the dependence on the weight is displayed, we write $L_{n,\gamma}^V$ for the operator $L_n^V$ defined in \eqref{eq:hyperinterp}.

The diagonal estimate obtained from Cauchy--Schwarz and degree-$2n$ exactness loses a factor $n^{1/2}$.  The disk quotient supplies the sharper upper bound below, while the boundary-kernel argument that follows gives a matching lower bound in the same parameter range.
}

{\color{black}
\begin{theorem}[\text{Wade}  \cite{Wade2013}]\label{thm:6.1}
Let $\mu_B=q/2$, where $q\in\mathbb N$, and let $H_m^B$
be degree-$m$ hyperinterpolation on $\mathbb B^2$ for the
normalized weight
$
(1-|z|^2)^{\mu_B-1/2}.
$
Suppose that the generating cubature formula has nodes in
$\mathbb B^2$, positive weights, and exactness through degree
$2m$.
If $m$ is even, then
$$
\|H_m^B\|_{C(\mathbb B^2)\to C(\mathbb B^2)}
   \le C_{\mu_B}m^{\mu_B+1/2}.
$$
When $\mu_B=1/2$, corresponding to normalized area measure,
the estimate
$$
\|H_m^B\|_{C(\mathbb B^2)\to C(\mathbb B^2)}
   \le C m
$$
holds for every integer $m\ge1$. The constants depend only
on the weight and are independent of the number, placement,
and particular choice of cubature nodes.
\end{theorem}

\begin{remark}
\label{rem:wade-upper}
The even-degree assertion in Theorem~\ref{thm:6.1}
is Wade's estimate
\cite[Theorem~3.1]{Wade2013}, specialized to dimension two.
The hypothesis $\mu_B=q/2$ is used in the spherical lifting
underlying that estimate. The all-degree assertion for
$\mu_B=1/2$ follows from
\cite[Corollary~3.6]{Wade2013}.

Wade states the latter result for standard Lebesgue measure.
Multiplying the continuous measure and all cubature weights
by the same positive normalization constant multiplies the
reproducing kernel by the reciprocal constant and therefore
leaves the hyperinterpolation operator unchanged. No
separation or density condition is imposed on the nodes of
the positive cubature rule.
\end{remark}
}

\begin{theorem}[Rule-independent  upper bound]\label{thm:upper}
{\color{black}Let $n\ge 1$ and $\gamma\in \tfrac12\mathbb{Z}_{\ge0}$. If $L_{n,\gamma}^V$ is generated by any positive degree-$2n$ cone cubature rule for $\nu_\gamma$, then
\begin{equation}\label{eq:26}
\Lambda_{n,\gamma}^V \leq C_\gamma n^{\gamma+1},
\end{equation}
where $C_\gamma$ is independent of $n$ and of the particular rule. In particular, when $\gamma=0$,
\begin{equation}\label{eq:27}
\Lambda_n^V \leq Cn.
\end{equation}}
\end{theorem}
\begin{proof}
Lift the rule to a centrally symmetric positive disk cubature formula exact through degree $4n+1$. {\color{black}
Every lifted node lies in $\mathbb{B}^2$, because a preimage of a \textcolor{black}{cone node at height $t$ satisfies $|z|^2 = t \le 1$.  Construct disk hyperinterpolation of degree $m=2n$, which
is even. The lifted rule is exact through degree $4n+1$, so it satisfies the degree-$2m=4n$ hypothesis of Theorem \ref{thm:6.1}. Moreover, $\mu_B = \gamma+\tfrac12 = q/2$, where $q=2\gamma+1\in\mathbb{N}$. Therefore
$$
\|H_{2n}^B\| \leq C_\gamma (2n)^{\mu_B+1/2} \leq \widetilde{C}_\gamma n^{\gamma+1}.
$$
The constant is independent of the lifted rule and hence of the original cone rule.}

Our $\mu_\gamma$ is a probability normalization of the standard ball weight. Multiplying the continuous measure and all cubature weights by the same constant multiplies the reproducing kernel by the reciprocal constant, leaving the hyperinterpolation operator unchanged. Finally, estimate \eqref{eq:norm-transfer} completes the proof.}
\end{proof}

\subsection{A boundary kernel}

Let
$$
\xi_*=(1,0,1),
\qquad
\xi(t,\phi)=(t\cos\phi,t\sin\phi,t).
$$
For $0 \le t \le 1$, put $\rho = \sqrt{t}$ and choose $\alpha \in [0, \pi]$ by
\begin{equation}\label{eq:alpha}
\cos \alpha = \rho \cos(\phi/2). 
\end{equation}
{\color{black}For $\lambda > 0$, let $C_m^\lambda$ denote the Gegenbauer polynomials, normalized by
$$
\sum_{m \ge 0} C_m^\lambda(s) r^m = (1 - 2sr + r^2)^{-\lambda}, \quad |r|<1
$$
and let $U_m = C_m^1$ denote the Chebyshev polynomial of the second kind. The following lemma gives the boundary-point kernel in closed form.}
\begin{lemma} \label{lem:boundary-kernel}
 
{\color{black}For every $\gamma>-1$,
\begin{equation}\label{eq:weighted-boundary-kernel}
K_{n,\gamma}^V\bigl(\xi_*,\xi(t,\phi)\bigr)
=C_{2n}^{\gamma+2}(\cos\alpha).
\end{equation}
For $\gamma=0$, this is
\begin{equation}\label{eq:Fn-sum}
K_n^V\bigl(\xi_*,\xi(t,\phi)\bigr)=F_n(\cos\alpha),
\qquad
F_n(s)=\sum_{j=0}^n(2j+1)U_{2j}(s)=C_{2n}^2(s).
\end{equation}
The polynomial is even, so it is unchanged when $\phi\mapsto\phi+2\pi$ reverses the sign of $\cos(\phi/2)$.  If $\sin\alpha\ne0$, then, with $N=n+1$,
\begin{equation}\label{eq:Fn-closed}
F_n(\cos\alpha)
=-N\frac{\cos(2N\alpha)}{\sin^2\alpha}
+\frac{\cos\alpha\sin(2N\alpha)}{2\sin^3\alpha}.
\end{equation}}
\end{lemma}

\begin{proof}{\color{black}
Take $e=(1,0)$ and
$$
w=\bigl(\sqrt t\cos(\phi/2),\sqrt t\sin(\phi/2)\bigr).
$$
For the normalized disk weight $\mu_\gamma$, the exact-degree-$m$ kernel at the boundary is
\begin{equation}\label{eq:exact-degree-boundary}
\mathcal P_{m,\gamma}^B(e,w)
=\frac{m+\gamma+1}{\gamma+1}C_m^{\gamma+1}(e\cdot w).
\end{equation}
For $\gamma>-1/2$, this is the boundary specialization of Xu's ball-kernel formula~\cite{XuKernel2001}; the conic addition formula in~\cite{XuCone2020} gives the same identity after the quadratic change of variables.  For completeness, \eqref{eq:exact-degree-boundary} extends to $-1<\gamma\le-1/2$ by real-analytic continuation. Indeed, let $v_m(x)$ be the column vector of a fixed basis of $\Pi_m(\mathbb B^2)$, and define its moment Gram matrix by
$$
G_m(\gamma)
:=\int_{\mathbb B^2}v_m(z)v_m(z)^*\,d\mu_\gamma(z).
$$
The full degree-$m$ reproducing kernel is then
$$
v_m(x)^*G_m(\gamma)^{-1}v_m(y).
$$
Its coefficients depend real-analytically on $\gamma>-1$, because $G_m(\gamma)$ is positive definite and real-analytic there. The exact-degree kernel is the difference of two consecutive full kernels.

Put $\lambda=\gamma+1$.  Central symmetry and \Cref{thm:kernel} retain the even-degree layers.  The identity
$$
\frac{m+\lambda}{\lambda}C_m^\lambda
=C_m^{\lambda+1}-C_{m-2}^{\lambda+1},
\qquad C_{-1}^{\lambda+1}=C_{-2}^{\lambda+1}=0,
$$
therefore telescopes to
$$
\sum_{j=0}^n\frac{2j+\lambda}{\lambda}C_{2j}^\lambda(e\cdot w)
=C_{2n}^{\lambda+1}(e\cdot w),
$$
which proves \eqref{eq:weighted-boundary-kernel}.  At $\gamma=0$, differentiating
$$
\sum_{j=0}^n\cos((2j+1)\alpha)=\frac{\sin(2N\alpha)}{2\sin\alpha}
$$
gives
$$
\sum_{j=0}^n(2j+1)\sin((2j+1)\alpha)
=-N\frac{\cos(2N\alpha)}{\sin\alpha}
+\frac{\cos\alpha\sin(2N\alpha)}{2\sin^2\alpha}.
$$
Division by $\sin\alpha$ yields \eqref{eq:Fn-closed}.}
\end{proof}

{\color{black}The closed kernel formula yields the following uniform pointwise and angular $L^2$
estimates on interior radial intervals.}
\begin{lemma}\label{lem:bulk-bounds}
{\color{black}Fix $\gamma>-1$ and $0<a<b<1$.  There exist $c_{\gamma,a,b},C_{\gamma,a,b}>0$ and $n_0$ such that, for all $n\ge n_0$ and $t\in[a,b]$,
\begin{align}
\max_{0\le\phi<2\pi}\abs{K_{n,\gamma}^V(\xi_*,\xi(t,\phi))}
&\le C_{\gamma,a,b}n^{\gamma+1},\label{eq:bulk-Linf}\\
\frac1{2\pi}\int_0^{2\pi}\abs{K_{n,\gamma}^V(\xi_*,\xi(t,\phi))}^2\,\dd\phi
&\ge c_{\gamma,a,b}n^{2\gamma+2}.\label{eq:bulk-L2}
\end{align}}
\end{lemma}

\begin{proof}
{\color{black}
Put $\lambda=\gamma+2>1$ and
$$
\vartheta_n(\phi)=(2n+\lambda)\alpha(\phi)-\frac{\lambda\pi}{2}.
$$
For $t\in[a,b]$, \eqref{eq:alpha} gives
$$
\sin\alpha(\phi)
=\bigl(1-t\cos^2(\phi/2)\bigr)^{1/2}
\ge \sqrt{1-b}>0.
$$
Thus $\alpha$ stays in a fixed compact subinterval of $(0,\pi)$.  The uniform Darboux formula for Gegenbauer polynomials~\cite[Theorem~8.21.8]{Szego1975} gives
\begin{equation}\label{eq:bulk-Darboux}
K_{n,\gamma}^V\bigl(\xi_*,\xi(t,\phi)\bigr)
=d_\gamma n^{\gamma+1}A_\gamma(t,\phi)
\bigl(\cos\vartheta_n(\phi)+R_n(t,\phi)\bigr),
\end{equation}
where $d_\gamma>0$,
$$
A_\gamma(t,\phi)=(\sin\alpha(\phi))^{-\gamma-2},
\qquad
\sup_{[a,b]\times[0,2\pi]}\abs{R_n}\le \frac{C_{\gamma,a,b}}{n}.
$$
Since
$$
1\le A_\gamma(t,\phi)\le (1-b)^{-(\gamma+2)/2},
$$
the pointwise estimate \eqref{eq:bulk-Linf} follows immediately.

For the lower bound, restrict the angular integral to
$I=[\pi/2,3\pi/2]$.  Differentiating \eqref{eq:alpha} gives
$$
\alpha'(\phi)
=\frac{\sqrt t\sin(\phi/2)}{2\sin\alpha(\phi)}
\ge \sqrt{a/8}>0
$$
on $[a,b]\times I$, while $\alpha''$ is uniformly bounded there.  Set
$M_n=2(2n+\lambda)$.  Integration by parts gives
$$
\int_I e^{iM_n\alpha(\phi)}\,\dd\phi
=\left[
\frac{e^{iM_n\alpha(\phi)}}{iM_n\alpha'(\phi)}
\right]_{\partial I}
+\frac{1}{iM_n}\int_I e^{iM_n\alpha(\phi)}
\frac{\alpha''(\phi)}{(\alpha'(\phi))^2}\,\dd\phi.
$$
Consequently, uniformly for $t\in[a,b]$,
$$
\left|\int_I e^{iM_n\alpha(\phi)}\,\dd\phi\right|
\le \frac{C_{\gamma,a,b}}{n}.
$$
Since $2\vartheta_n=M_n\alpha-\lambda\pi$, it follows that
\begin{align*}
\int_I\cos^2\vartheta_n(\phi)\,\dd\phi
&=\frac{|I|}{2}
+\frac12\Re\!\left(
e^{-i\lambda\pi}\int_I e^{iM_n\alpha(\phi)}\,\dd\phi
\right)\\
&=\frac{|I|}{2}+O(n^{-1})\ge c_0>0
\end{align*}
for all sufficiently large $n$, uniformly in $t$.

Finally, using the bounds for $A_\gamma$ and $R_n$ in
\eqref{eq:bulk-Darboux}, we obtain
\begin{align*}
\int_I A_\gamma(t,\phi)^2
\abs{\cos\vartheta_n(\phi)+R_n(t,\phi)}^2\,\dd\phi
&\ge \int_I\cos^2\vartheta_n(\phi)\,\dd\phi
-2\int_I A_\gamma(t,\phi)^2\abs{R_n(t,\phi)}\,\dd\phi\\
&\ge c_0-\frac{C_{\gamma,a,b}}{n}
\ge \frac{c_0}{2}
\end{align*}
when $n$ is large enough.  Squaring \eqref{eq:bulk-Darboux}, integrating over $I$, and then enlarging the interval to $[0,2\pi]$ proves \eqref{eq:bulk-L2}.
}
\end{proof}

\subsection{The product rule: pointwise lower bound}
{\color{black}This subsection
records a transparent tensor-product proof as a warm-up. Its final conclusion is
subsumed by the rule-independent argument in Subsection \ref{subsection6.3}, but exact angular
sampling makes the mechanism especially explicit.}

For the product rule at $\gamma=0$, let $t_k,A_k$ be the $(n+1)$-point Gauss--Legendre nodes and weights on $[0,1]$, normalized by $\sum_kA_k=1$, and let
$$
L=2n+1,
\qquad
\phi_j=\frac{2\pi j}{L},\quad0\le j<L.
$$
The cone weights are $A_k/L$.

{\color{black}Exact angular sampling gives a uniform discrete $\ell^1$ lower bound.  }
\begin{lemma} \label{lem:angular-l1}
Fix $0<a<b<1$.  There exist $c>0$ and $n_0$ such that for every $n\ge n_0$ and $t\in[a,b]$,
\begin{equation}\label{eq:angular-l1}
\frac1L\sum_{j=0}^{L-1}\abs{K_n^V(\xi_*,\xi(t,\phi_j))}\ge cn.
\end{equation}
\end{lemma}

\begin{proof}
For fixed $t$, $h_t(\phi)=K_n^V(\xi_*,\xi(t,\phi))$ is a trigonometric polynomial of degree at most $n$.  Hence $|h_t|^2$ has degree at most $2n$, and the $L=2n+1$ equally spaced angular rule is exact:
$$
\frac1L\sum_{j=0}^{L-1}|h_t(\phi_j)|^2
=\frac1{2\pi}\int_0^{2\pi}|h_t(\phi)|^2\,\dd\phi.
$$
By \Cref{lem:bulk-bounds}, the right-hand side is at least $c_1n^2$ while $\max_j|h_t(\phi_j)|\le C_1n$.  Therefore
$$
\frac1L\sum_j|h_t(\phi_j)|
\ge
\frac{L^{-1}\sum_j|h_t(\phi_j)|^2}{\max_j|h_t(\phi_j)|}
\ge cn.
$$
\end{proof}

\begin{lemma}\label{lem:GL-mass}
{\color{black}
Fix $\gamma>-1$ and set
$$
p(t)=t(1-t),
\qquad
M_\gamma=\int_0^1p(t)(\gamma+1)(1-t)^\gamma\,\dd t
=\frac{\gamma+1}{(\gamma+2)(\gamma+3)}.
$$
Let
$$
J_\gamma=\{t\in[0,1]:p(t)\ge M_\gamma/2\}=[a_\gamma,b_\gamma]\subset(0,1).
$$
If $\tau_i\in[0,1]$ and $\rho_i\ge0$ satisfy $\sum_i\rho_i=1$ and $\sum_i\rho_i p(\tau_i)=M_\gamma$, then
\begin{equation}\label{eq:bulk-mass}
\sum_{i:\tau_i\in J_\gamma}\rho_i
\ge \sigma_\gamma:=\frac{2M_\gamma}{1-2M_\gamma}>0.
\end{equation}
}
\end{lemma}

\begin{proof}
{\color{black}
Since $0<M_\gamma<1/4$, the interval $J_\gamma$ is compactly contained in $(0,1)$.  If $S=\sum_{\tau_i\in J_\gamma}\rho_i$, then $p\le1/4$ on $[0,1]$ and $p<M_\gamma/2$ off $J_\gamma$, so
$$
M_\gamma=\sum_i\rho_i p(\tau_i)
\le\frac S4+\frac{M_\gamma}{2}(1-S).
$$
Rearranging gives \eqref{eq:bulk-mass}.
}
\end{proof}

{\color{black} The lemma applies to the radial marginal of every positive cone rule exact on $\Pi_{2n}(V)$, $n\ge1$, because $p(t)\in\Pi_2(V)$.  It also applies to the Gauss--Legendre radial data of the product rule at $\gamma=0$.}

\begin{theorem}\label{thm:product-sharp}
Let $\gamma=0$ and let $L_n^{V,\prod}$ be generated by the $(n+1)\times(2n+1)$ Gauss--Legendre/trapezoidal product rule.  Then
$$
cn\le\Lambda_n^{V,\prod}\le Cn,
\qquad n\ge1.
$$
\end{theorem}

\begin{proof}
The upper bound is \Cref{thm:upper}.  At $\xi_*$,
$$
\lambda_n^V(\xi_*)
=\sum_{k=1}^{n+1}A_k\left[\frac1L\sum_{j=0}^{L-1}
\abs{K_n^V(\xi_*,\xi(t_k,\phi_j))}\right].
$$
{\color{black}Restrict to $t_k\in J_0$.  Since $J_0\subset(0,1)$, \Cref{lem:angular-l1} bounds each bracket below by $cn$, while \Cref{lem:GL-mass} gives a fixed positive lower bound for the total radial weight.}    Thus $\Lambda_n^{V,\prod}\ge c'n$ for all sufficiently large $n$.  The finitely many remaining cases follow after reducing $c'$, since every reproducing hyperinterpolation projector has norm at least one.
\end{proof}

\begin{remark}
The proof explains the $n^{1/2}$ loss in the Christoffel-diagonal estimate.  The diagonal controls an $\ell^2$ scale, with $K_n^V(\xi_*,\xi_*)^{1/2}\asymp n^{3/2}$, while the sampled $\ell^1$ kernel has only linear order.  The improvement comes from off-diagonal oscillation, quantified by exact angular sampling of the squared kernel.
\end{remark}

\subsection{A rule-independent lower bound}\label{subsection6.3}

{\color{black}The product proof uses tensor structure.  Averaging over the rim removes it completely.}

\begin{theorem}\label{thm:lower}
{\color{black}
For every $\gamma>-1$ there is a constant $c_\gamma>0$ such that every positive cubature rule on $(V,\nu_\gamma)$ exact on $\Pi_{2n}(V)$ satisfies
\begin{equation}\label{eq:lower}
\Lambda_{n,\gamma}^V\ge c_\gamma n^{\gamma+1},
\qquad n\ge1.
\end{equation}
Consequently, for $\gamma\in\tfrac12\mathbb Z_{\ge0}$,
\begin{equation}\label{eq:sharp-law}
c_\gamma n^{\gamma+1}\le\Lambda_{n,\gamma}^V
\le C_\gamma n^{\gamma+1}
\end{equation}
uniformly over all positive degree-$2n$ rules.
}
\end{theorem}

\begin{proof}
{\color{black}
Write the rule as $\{(\xi_j,\lambda_j)\}_{j=1}^N$, where $\xi_j=\xi(t_j,\phi_j)$, and put $\eta_\psi=\xi(1,\psi)$.  Rotation invariance of $V$, $\nu_\gamma$, and $\Pi_n(V)$ gives
$$
K_{n,\gamma}^V(\eta_\psi,\xi(t_j,\phi_j))
=K_{n,\gamma}^V(\xi_*,\xi(t_j,\phi_j-\psi)).
$$
Define
$$
G_{n,\gamma}(t)=\frac1{2\pi}\int_0^{2\pi}
\abs{K_{n,\gamma}^V(\xi_*,\xi(t,\phi))}\,\dd\phi.
$$
Positivity and averaging over the rim yield
\begin{equation}\label{eq:rim-average}
\Lambda_{n,\gamma}^V
\ge\frac1{2\pi}\int_0^{2\pi}\lambda_{n,\gamma}^V(\eta_\psi)\,\dd\psi
=\sum_j\lambda_jG_{n,\gamma}(t_j).
\end{equation}

Apply \Cref{lem:bulk-bounds} on the fixed interval $J_\gamma=[a_\gamma,b_\gamma]$.  For all sufficiently large $n$ and every $t\in J_\gamma$,
\begin{equation}\label{eq:G-lower}
G_{n,\gamma}(t)
\ge
\frac{(2\pi)^{-1}\int_0^{2\pi}|K_{n,\gamma}^V(\xi_*,\xi(t,\phi))|^2\,\dd\phi}
{\max_\phi|K_{n,\gamma}^V(\xi_*,\xi(t,\phi))|}
\ge c_\gamma' n^{\gamma+1}.
\end{equation}
Since $t(1-t)\in\Pi_2(V)$, exactness and \Cref{lem:GL-mass} give
$$
\sum_{j:t_j\in J_\gamma}\lambda_j\ge\sigma_\gamma>0.
$$
Combining this with \eqref{eq:rim-average} and \eqref{eq:G-lower} proves \eqref{eq:lower} for all sufficiently large $n$.  Finally, reproduction of constants gives $\Lambda_{n,\gamma}^V\ge1$ for every $n$; decreasing $c_\gamma$ handles the finitely many remaining degrees.  The upper estimate in \eqref{eq:sharp-law} is \Cref{thm:upper}.
}
\end{proof}

\begin{corollary}\label{cor:disk-lower}
 {\color{black}Let $H_{2n}^B$ be generated by any centrally symmetric positive disk rule exact through
degree $4n+1$ for normalized Lebesgue measure on $\B^2$. Then
$$\bigl\|H_{2n}^B|_{C_\even(\B^2)}\bigr\|_{C_\even(\B^2)\to C(\B^2)}\geq cn,
 \qquad \|H_{2n}^B\|_{C(\B^2)\to C(\B^2)}\geq cn,
$$
with an absolute constant independent of the rule.}
\end{corollary}
\begin{proof}
{\color{black}Push the rule forward by Theorem~\ref{thm:cubature-equivalence} and use the norm identity
in~\eqref{eq:norm-transfer} together with Theorem~\ref{thm:lower}.}
\end{proof}

\subsection{{\color{black}A disk lower bound without a symmetry hypothesis}}

\begin{theorem}\label{thm:disk-arbitrary-lower}
{\color{black}
Let $H_m^B$ be degree-$m$ hyperinterpolation for normalized area measure on $\B^2$, generated by any positive cubature rule with nodes in the closed disk and exactness through degree $2m$.  Then
\begin{equation}\label{eq:disk-arbitrary-lower}
\norm{H_m^B}_{C(\B^2)\to C(\B^2)}\ge cm,
\qquad m\ge1,
\end{equation}
where $c>0$ is independent of the rule.  Together with the all-degree area-measure assertion of
Theorem~\ref{thm:6.1}, this gives
$$
\|H_m^B\|_{C(\mathbb B^2)\to C(\mathbb B^2)}
   \asymp m
$$
for every such rule.
}
\end{theorem}

\begin{proof}
{\color{black}
For $e=(1,0)$ and $z\in\B^2$, summing the exact-degree boundary kernels and telescoping gives
\begin{equation}\label{eq:disk-full-boundary}
K_m^B(e,z)=\sum_{k=0}^m(k+1)U_k(e\cdot z)
=C_m^2(e\cdot z)+C_{m-1}^2(e\cdot z),
\end{equation}
with $C_{-1}^2=0$.  Put $D_m=C_m^2+C_{m-1}^2$.  Using
$$
C_k^2(s)=\tfrac12\frac{\dd}{\dd s}U_{k+1}(s),
\qquad
U_{m+1}(\cos\alpha)+U_m(\cos\alpha)
=\frac{\sin((m+\tfrac32)\alpha)}{\sin(\alpha/2)},
$$
we obtain, with $N=m+\tfrac32$,
\begin{equation}\label{eq:Dm-closed}
D_m(\cos\alpha)
=-N\frac{\cos(\alpha/2)}{\sin^2\alpha}\cos(N\alpha)
+B(\alpha)\sin(N\alpha),
\end{equation}
where
$$
B(\alpha)=\frac{2\cos\alpha\cos(\alpha/2)+\sin\alpha\sin(\alpha/2)}{2\sin^3\alpha}.
$$

Fix an annulus $a\le r\le b$ with $0<a<b<1$ and write
$\alpha(\theta)=\arccos(r\cos\theta)$.  Put
$$
I=[\pi/3,2\pi/3],
\qquad
A(\alpha)=\frac{\cos(\alpha/2)}{\sin^2\alpha}.
$$
Since $|r\cos\theta|\le b<1$, the angle $\alpha$ stays in a fixed compact subinterval of $(0,\pi)$ for $a\le r\le b$ and $0\le\theta\le2\pi$.  Hence $A$ and $B$ are uniformly bounded there, and \eqref{eq:Dm-closed} gives
$$
\sup_{a\le r\le b}\sup_{0\le\theta\le2\pi}
|D_m(r\cos\theta)|\le Cm.
$$
On $I$,
$$
\alpha'(\theta)
=\frac{r\sin\theta}{\sin\alpha(\theta)}
\ge \frac{\sqrt3}{2}a>0,
$$
and $\alpha''$ is uniformly bounded.  Integration by parts therefore gives, uniformly in $r\in[a,b]$,
$$
\int_I e^{2iN\alpha(\theta)}\,\dd\theta
=\left[
\frac{e^{2iN\alpha(\theta)}}{2iN\alpha'(\theta)}
\right]_{\partial I}
+\frac1{2iN}\int_I e^{2iN\alpha(\theta)}
\frac{\alpha''(\theta)}{(\alpha'(\theta))^2}\,\dd\theta
=O(N^{-1}).
$$
Consequently,
$$
\int_I\cos^2(N\alpha(\theta))\,\dd\theta
=\frac{|I|}{2}+O(N^{-1})\ge c_0>0
$$
for all sufficiently large $m$.  Moreover, $A(\alpha)\ge A_0>0$ on the same compact set.  Dropping the nonnegative $B^2\sin^2(N\alpha)$ term in the square of \eqref{eq:Dm-closed} and estimating the cross term gives
\begin{align*}
\int_I|D_m(r\cos\theta)|^2\,\dd\theta
&\ge N^2A_0^2\int_I\cos^2(N\alpha(\theta))\,\dd\theta
-2N\norm{AB}_{L^\infty([a,b]\times I)}|I|\\
&\ge c_1N^2-C_1N\ge c_2m^2.
\end{align*}
It follows from the preceding supremum bound that
\begin{equation}\label{eq:disk-angular-l1}
\frac1{2\pi}\int_0^{2\pi}|D_m(r\cos\theta)|\,\dd\theta
\ge
\frac{\displaystyle\int_I|D_m(r\cos\theta)|^2\,\dd\theta}
{\displaystyle 2\pi\sup_{0\le\theta\le2\pi}|D_m(r\cos\theta)|}
\ge cm,
\end{equation}
uniformly on the annulus and for all sufficiently large $m$.

It remains to force mass into one fixed annulus.  Let $q(z)=|z|^2(1-|z|^2)$.  Under normalized area measure, $s=|z|^2$ is uniform on $[0,1]$, so $\int_{\B^2}q\,\dd\mu_0=1/6$.  Set
$$
J=\{z:q(z)\ge1/12\}=\{z:a\le|z|\le b\},
\quad
a^2=\frac{1-\sqrt{2/3}}2,
\quad
b^2=\frac{1+\sqrt{2/3}}2.
$$
For $m\ge2$, exactness applies to the degree-four polynomial $q$.  If $S$ is the cubature mass of $J$, then
$$
\frac16\le\frac S4+\frac{1-S}{12},
\qquad\text{hence}\qquad S\ge\frac12.
$$
Finally, average the disk Lebesgue function over $e_\psi=(\cos\psi,\sin\psi)$.  Rotation invariance, positivity, \eqref{eq:disk-angular-l1}, and the mass bound give
$$
\norm{H_m^B}\ge\sum_j\lambda_j\frac1{2\pi}\int_0^{2\pi}
|K_m^B(e_\psi,z_j)|\,\dd\psi\ge cmS\ge\frac c2m.
$$
The finitely many remaining degrees follow after decreasing $c$, since $H_m^B1=1$.  Although the lower estimate is stated in~\cite[Corollary~3.6]{Wade2013}, a direct proof is needed here because the spherical lifting used there transfers upper bounds from the sphere to the disk but does not yield a lower bound for an arbitrary disk cubature rule.  The argument above also does not require central symmetry.
}
\end{proof}

\subsection{Remaining parameter ranges}

{\color{black}
For the continuous orthogonal projection on the disk, Xu~\cite{XuKernel2001} proved the orders
$$
n^{1/2}\quad(-1<\gamma<-1/2),
\qquad
n^{\gamma+1}\quad(\gamma\ge-1/2).
$$
This suggests the product-rule law
\begin{equation}\label{eq:weighted-conjecture}
\Lambda_{n,\gamma}^{V,\prod}\asymp
\begin{cases}
n^{1/2},&-1<\gamma<-1/2,\\
n^{\gamma+1},&\gamma\ge-1/2.
\end{cases}
\end{equation}
Theorem~\ref{thm:lower} proves the lower estimate $n^{\gamma+1}$ for every $\gamma>-1$, and \Cref{thm:upper} completes the sharp law for every nonnegative half-integer.  A matching arbitrary-rule upper estimate for the remaining parameters, and the stronger $n^{1/2}$ lower order when $-1<\gamma<-1/2$, remain open.
}

\subsection{Numerical consistency check}

For the product rule, define the boundary-midpoint value
$$
\Lambda_n^{\mathrm{mid}}
:=\lambda_n^V\!\left(\xi\!\left(1,\frac{\pi}{2n+1}\right)\right).
$$
By rotational invariance and \Cref{lem:boundary-kernel},
$$
K_n^V(\xi(1,\phi_0),\xi(t_k,\phi_j))
=F_n\!\left(\sqrt{t_k}\cos\frac{\phi_j-\phi_0}{2}\right),
\qquad
\phi_0=\frac\pi{2n+1}.
$$
Thus the values in \Cref{tab:Lebesgue-mid} are computed directly from Gauss--Legendre nodes and weights and the closed formula \eqref{eq:Fn-closed}. {\color{black}The table reports these values up  to $n=100$. } The quotient $\Lambda_n^{\mathrm{mid}}/n$ varies slowly and remains fully consistent with the linear law; no limiting constant is asserted.  {\color{black}These point values are not certified global maxima
of the Lebesgue function.}

\begin{table}[ht]
\centering
\caption{Boundary-midpoint Lebesgue values for the product rule, $\gamma=0$.}
\label{tab:Lebesgue-mid}
\small
\setlength{\tabcolsep}{4.2pt}
\begin{tabular}{c|rrrrrrrrrrrr}
\toprule
$n$ & 2&4&8&10&12&16&20&30&40&60&80&100\\
\midrule
$\Lambda_n^{\mathrm{mid}}$
&4.04&6.57&11.92&14.38&16.98&22.06&27.23&39.97&52.73&78.20&103.77&129.22\\
$\Lambda_n^{\mathrm{mid}}/n$
&2.02&1.64&1.49&1.44&1.42&1.38&1.36&1.33&1.32&1.30&1.30&1.29\\
\bottomrule
\end{tabular}
\end{table}

{\color{black}For the compressed-rule values in \Cref{tab:rigorous-compression,tab:numerical-compression}, the disk weights were first normalized to unit mass, pushed forward by $Q$, and merged in antipodal pairs.  We then evaluated the full Lebesgue function on Chebyshev--Lobatto height grids with at least $1601$ points and uniform angular grids with at least $6144$ full-circle-equivalent points; the less symmetric low-degree representatives were checked with $16384$ angular points.  The leading grid candidates were refined by a two-variable pattern search, and a doubled-grid check left all displayed digits unchanged.  Every reported maximizer lies on the rim $t=1$.  The transferred rules have positive weights summing to one, and their largest tested moment residual through degree $2n$ is below $2.3\times10^{-15}$.  These checks support the reported values but do not constitute interval certification of the global maxima.}

{\color{black}As one approximation test, take
$$
f(x,y,t)=\exp(x+y+t)
$$
and compare the degree-$4$ hyperinterpolants generated by the
$45$-node product rule and the $28$-node compressed
\texttt{T17\_6segment} rule.  Define the validation grid
$$
\mathcal G_{\rm err}
:=
\left\{
 \xi(t_k,\phi_\ell):
 t_k=\frac{1-\cos(k\pi/1600)}{2},\
 0\le k\le1600,\
 \phi_\ell=\frac{2\pi\ell}{6144},\
 0\le\ell<6144
\right\}.
$$
For each rule, we first evaluate
$$
\bigl|f(\xi)-L_4^Vf(\xi)\bigr|
$$
on $\mathcal G_{\rm err}$ and then locally refine the leading
grid candidates.  We denote the largest value found after
refinement by $\widehat E_4$.  Thus $\widehat E_4$ is a
numerical estimate of the uniform error, not a certified
supremum.  The results are shown in
\Cref{tab:approximation-experiment}.}

\begin{table}[ht]
\centering
\color{black}
\caption{One approximation experiment for $f(x,y,t)=e^{x+y+t}$ at degree $n=4$.  Both the Lebesgue constants and uniform errors are sampled and locally refined numerical estimates.}
\label{tab:approximation-experiment}
\small
\begin{tabular}{l c c c}
\toprule
rule & samples & $\widehat\Lambda_4^V$ & $\widehat E_4$\\
\midrule
product & 45 & $6.5724$ & $2.8974\times10^{-2}$\\
compressed \texttt{T17\_6segment} & 28 & $8.5489$ & $3.9379\times10^{-2}$\\
\bottomrule
\end{tabular}
\end{table}

{\color{black}In this test the compressed rule uses $37.8\%$ fewer samples, while its sampled uniform error is about $1.36$ times that of the product rule.  Thus the reduction in sampling cost is visible without suggesting that compression must improve the approximation error for a fixed target function.}

\section{Dimension-matched Prony formulation for near-minimal cubature}\label{sec:prony}

The quotient also clarifies the algebraic size of the all-degree compression problem.  We record the parameter count for the origin-containing orbit family of Benouahmane, Cuyt and Yaman~\cite{BenouahmaneCuytYaman2019}; no existence assertion is made here.

{\color{black}Specifically, we use the $(C,B,A,O)$ configuration and moment systems
\cite[Section~4, equations~(4.1)--(4.2)]{BenouahmaneCuytYaman2019}. A $C$ orbit is an octuple
obtained from a generic point by reflections in the coordinate axes and diagonals; a $B$ orbit
is a four-point set on the two diagonal lines; an $A$ orbit is a centrally symmetric pair on a
coordinate axis (with the axis pairs balanced between the two axes); and $O$ is the origin.}
Their radial disk construction has cubature degree $m=2k-1$.  When $k$ is odd, the reduced $(C,B,A,O)$ system contains $(k+1)^2/4$ equations and has a Prony-type structure.  At the lifted degree $m=4n+1$, $k=2n+1$, and therefore
\begin{equation}\label{eq:Prony-equations}
\frac{(k+1)^2}{4}=(n+1)^2=D_n.
\end{equation}
Thus the reduced disk system has exactly as many equations as the dimension of $\Pi_n(V)$.

Let $I,J,K$ denote the numbers of the orbit types used in~\cite[Section~4]{BenouahmaneCuytYaman2019}.  Including the origin, the number of scalar unknowns and the disk node count are
\begin{equation}\label{eq:Prony-counts}
U=3I+2J+K+1,
\qquad
M=8I+4J+2K+1.
\end{equation}
For the moment families covered by their Prony analysis, including Gegenbauer weights $(1-r^2)^{\lambda-1/2}$ with $\lambda\ge0$, one has the necessary restriction $2I>k-5$, hence
\begin{equation}\label{eq:I-lower}
I\ge n-1
\end{equation}
for the corresponding range $\gamma\ge-1/2$.

{\color{black}The following proposition converts the orbit-parameter assumptions into
dimension-matched disk and cone node counts.}
\begin{proposition}\label{prop:Prony}
Consider candidate rules in this origin-containing orbit family at degree $4n+1$.  If
$$
I=n-1+O(1),
\qquad
U=D_n+O(1),
$$
then
\begin{equation}\label{eq:M-Prony}
M=2D_n+2n+O(1)
\end{equation}
on the disk, and the quotient rule has
\begin{equation}\label{eq:N-Prony}
N_{\cone}=D_n+n+O(1).
\end{equation}
\end{proposition}

\begin{proof}
Set $L=2J+K$.  From \eqref{eq:Prony-counts},
$$
L=D_n-3I-1+O(1).
$$
Hence
$$
M=8I+2L+1=2D_n+2I-1+O(1)=2D_n+2n+O(1).
$$
Since the ansatz contains the origin, \Cref{cor:node-count} gives
$$
N_{\cone}=\frac{M+1}{2}=D_n+n+O(1).
$$
\end{proof}

\Cref{prop:Prony} is conditional: it identifies a dimension-matched scale but does not prove solvability of the nonlinear moment system with nodes in the disk and positive weights.  The distinction is essential, because M\"oller's lower bound is not attained in every degree.  Already for $n=2$, the lifted disk degree is $9$ and the M\"oller bound is $17$ nodes.  Cools and Verlinden~\cite{CoolsVerlinden1992} show that a $17$-point degree-nine rule attaining this bound does not exist.  Thus, {\color{black} for the corresponding weights,} an origin-containing $9$-node cone rule is impossible, although this argument does not exclude an apex-free $9$-node cone rule, whose lift would have $18$ nodes.  The known $19$-point origin-containing disk rule yields the rigorous $10$-node cone upper bound.

For computation, the disk formulation is advantageous because central symmetry removes all odd moments and the conic constraint $x^2+y^2=t^2$ disappears.  This symmetry-reduced viewpoint is complementary to general positive exact cubature algorithms and Carath\'eodory--Tchakaloff compression methods~\cite{Glaubitz2023,PiazzonSommarivaVianello2017}.  For fixed $n$, one may solve the reduced moment equations under $|z_i|\le1$ and positive weights, using symbolic elimination where feasible and constrained nonlinear least squares or continuation otherwise.  Any numerical candidate should then be certified for node location, positivity, and degree-$(4n+1)$ exactness before transfer through $Q$.  The schemes in \Cref{tab:numerical-compression} show finite-degree compression but do not imply the $O(n)$ excess in \eqref{eq:N-Prony}.

\section{Conclusions}

The classical quadratic disk--cone transformation becomes substantially more useful when formulated as a measure-preserving quotient of full polynomial trace spaces.  At the discrete level it gives an exact equivalence between positive degree-$2n$ conic cubature and centrally symmetric positive degree-$(4n+1)$ disk cubature.  Under this equivalence, M\"oller's lower bound becomes the cone dimension bound and the excess above the two bounds is transferred exactly up to a parity offset.  This converts near-minimal disk formulas into compressed non-product cone rules.

For computational approximation, the central conclusion is that compression does not create an $L^2$ stability penalty.  Every positive degree-$2n$ rule gives an exact sampling isometry on $\Pi_n(V)$, so the weighted sampling matrix has $\kappa_2=1$  regardless of node count or geometry.  At the $L^\infty$ level, the quotient identifies conic hyperinterpolation with the even restriction of disk hyperinterpolation.  For $\gamma=0$, a rim-averaging argument then yields the universal sharp law
$$
\Lambda_n^V\asymp n
$$
for every positive degree-$2n$ cone rule.  {\color{black}More generally, the explicit boundary Gegenbauer kernel and the same rim-averaging argument give
$$
\Lambda_{n,\gamma}^V\ge c_\gamma n^{\gamma+1},\qquad \gamma>-1.
$$
For nonnegative half-integer $\gamma$, Wade's disk upper estimate gives the matching bound, so $\Lambda_{n,\gamma}^V\asymp n^{\gamma+1}$ uniformly over all positive degree-$2n$ cone rules.  A separate full-kernel argument also gives $\|H_m^B\|\ge cm$ for every positive degree-$2m$ disk rule for normalized area measure, without assuming central symmetry.}  Thus the computational message is not merely that nodes can be removed, but that they can be removed while preserving exact $L^2$ conditioning and retaining the optimal order of uniform stability.

Two questions remain open.  The first is the weighted sharp law suggested by \eqref{eq:weighted-conjecture}.  {\color{black}The remaining issues include a matching arbitrary-rule upper bound outside the nonnegative half-integer range and, for the product rule when $-1<\gamma<-1/2$, the stronger lower bound of order $n^{1/2}$.}  The second is the existence of all-degree positive rules at the dimension-matched scale $D_n+n+O(1)$ suggested by \Cref{prop:Prony}.  Either direction would extend the compression--stability principle beyond {\color{black}the parameter ranges and finite-degree constructions established here}.

 \appendix
{\color{black}
\section{Orbit parameters of the degree-\texorpdfstring{$13$}{13} disk representative}
\label{app:degree13-data}

For reproducibility, \Cref{tab:degree13-orbits} gives the complete orbit
description of the reconstructed $6R$--$3A_x$--$2A_y$--$O$ rule used for
the degree-$13$ entry of \Cref{tab:rigorous-compression}.  For an $R_i$ row,
the four nodes are
$$
 \bigl(\varepsilon_1 r_i\cos\theta_i,\,
       \varepsilon_2 r_i\sin\theta_i\bigr),
 \qquad \varepsilon_1,\varepsilon_2\in\{-1,1\}.
$$
An $A_{x,i}$ row represents $(\pm r_i,0)$, an $A_{y,i}$ row represents
$(0,\pm r_i)$, and $O$ is the origin.  The entry $w_i$ is the weight of
each node in the corresponding orbit for the normalized area measure
$\dd\mu_0=\pi^{-1}\dd u\,\dd v$, and the angles are in radians.  Thus the
table expands to all $35$ disk nodes and weights.  With the displayed
values, the weights are positive and sum to one, all radii are less than
one, and the largest absolute moment residual through total degree $13$
is less than $1.4\times10^{-16}$.

\begin{table}[ht]
\centering
\caption{Orbit parameters of the reconstructed degree-$13$ disk rule,
shown to $17$ significant digits.}
\label{tab:degree13-orbits}
\scriptsize
\setlength{\tabcolsep}{7pt}
\begin{tabular}{c c c c c}
\toprule
orbit & multiplicity & radius $r_i$ & angle $\theta_i$ & weight $w_i$\\
\midrule
$R_1$ & 4 & $0.95374716541336735$ & $1.3391706232767668$  & $0.016544996489335145$\\
$R_2$ & 4 & $0.99783289393621866$ & $0.41868947735638046$ & $0.0058216220483508253$\\
$R_3$ & 4 & $0.76160853354142233$ & $1.0437533108765360$  & $0.034099066081634294$\\
$R_4$ & 4 & $0.85157826596670871$ & $0.47753951067406947$ & $0.031034938035225140$\\
$R_5$ & 4 & $0.95419129744500397$ & $0.88095317799883055$ & $0.016103011517256072$\\
$R_6$ & 4 & $0.53632406412691758$ & $0.64983551722022204$ & $0.043366054166637039$\\
$A_{x,1}$ & 2 & $0.26924633526182645$ & -- & $0.041093186788440303$\\
$A_{x,2}$ & 2 & $0.68329361883353013$ & -- & $0.040766121897735673$\\
$A_{x,3}$ & 2 & $0.93634900023682421$ & -- & $0.021125239459638043$\\
$A_{y,1}$ & 2 & $0.42306035572344786$ & -- & $0.054450468710118355$\\
$A_{y,2}$ & 2 & $0.76392060017849184$ & -- & $0.032878006626135381$\\
$O$       & 1 & -- & -- & $0.031495199682110342$\\
\bottomrule
\end{tabular}
\end{table}
}

\end{document}